\documentclass[12pt]{amsart}
\usepackage{amssymb}
\usepackage{graphicx}
\usepackage[cmtip,all]{xy}
\usepackage{bbm}
\usepackage{comment}
\usepackage{mathrsfs}
\usepackage{amsmath}

\def\DefineSymbol#1#2{\newcommand{#1}{{\mathrm {#2}}}}
\def\DefineCategory#1#2{\newcommand{#1}{{\mathbf {#2}}}}

\theoremstyle{plain}
	\newtheorem{theorem}{Theorem}[section]
	\newtheorem{lemma}[theorem]{Lemma}
	\newtheorem{proposition}[theorem]{Proposition}
	\newtheorem{corollary}[theorem]{Corollary}
	\newtheorem{construction}[theorem]{Construction}
\theoremstyle{definition}
	\newtheorem{definition}[theorem]{Definition}
	\newtheorem{example}[theorem]{Example}

\theoremstyle{remark}
	\newtheorem{remark}[theorem]{Remark}
	\numberwithin{equation}{section}

\DefineSymbol{\pr}{pr}
\DefineSymbol{\id}{id}
\DefineSymbol{\const}{const}
\DefineSymbol{\op}{op}
\DefineSymbol{\diag}{diag}
\DefineSymbol{\GL}{GL}
\DefineSymbol{\SL}{SL}
\DefineSymbol{\Int}{Int}
\DefineSymbol{\Res}{Res}
\DefineSymbol{\conj}{conj}
\DefineSymbol{\ind}{ind}

\let\Im\relax
\DeclareMathOperator{\Im}{Im}

\DeclareMathOperator{\Gal}{Gal}

\DeclareMathOperator{\Norm}{Norm}
\DeclareMathOperator{\Cent}{Cent}
\DeclareMathOperator{\Nm}{Nm}

\DefineCategory{\Set}{Set}
\DefineCategory{\Ab}{Ab}
\DefineCategory{\Mod}{Mod}
\DefineCategory{\Alg}{Alg}
\DefineCategory{\Ch}{Ch}
\DefineCategory{\Mon}{Mon}
\DefineCategory{\CMon}{CMon}
\DefineCategory{\PSh}{PSh}
\DefineCategory{\Sh}{Sh}

\newcommand{\Abb}{\mathbb{A}}

\newcommand{\Cbb}{\mathbb{C}}

\newcommand{\Gbb}{\mathbb{G}}
\newcommand{\Hbb}{\mathbb{H}}

\newcommand{\Rbb}{\mathbb{R}}

\newcommand{\Zbb}{\mathbb{Z}}

\newcommand{\Qcal}{\mathcal{Q}}

\newcommand{\Zcal}{\mathcal{Z}}

\calclayout

\begin{document}

\begin{abstract}
	We prove a general formula that relates the parity of the Langlands parameter of a conjugate self-dual discrete series representation of $\GL_n$ to the parity of its Jacquet-Langlands image.
	It gives a generalization of a partial result by Mieda concerning the case of invariant $1/n$ and supercuspidal representations. It also gives a variation of the result on the self-dual case by Prasad and Ramakrishnan.
\end{abstract}

\title[Parity of conjugate self-dual representations]{PARITY OF CONJUGATE SELF-DUAL REPRESENTATIONS OF INNER FORMS OF $\GL_n$ over $p$-adic fields}
\author{Yugo Takanashi}
\address{Graduate School of Mathematical Sciences, University of Tokyo, 3-8-1 Komaba, Meguro-Ku, Tokyo 153-8914, Japan}
\email{tknashi@ms.u-tokyo.ac.jp}
\maketitle

\setcounter{tocdepth}{1}	% hide subsections in the table of contents
\tableofcontents

\section*{Introduction}
    Let $F$ be a $p$-adic field. We have two parametrizations of the isomorphism classes of the discrete series representations of $\GL_n(F)$, the local Jacquet-Langlands correspondence and the local Langlands correspondence. 
    If we fix an inner form $\GL_m(D)$ of $\GL_n(F)$, the former associates each discrete series representation of $\GL_n(F)$ with a discrete series representation of the inner form in such a way that a certain character relation holds (see \cite{DKVcentalsimple, Rog}).
    The latter is a Galois-theoretic parametrization which associates each discrete series representation of $\GL_n(F)$ with a discrete Langlands parameter, that is, an irreducible $n$-dimensional representation of the Weil-Deligne group $W_F \times \SL_2(\Cbb)$, where $W_F$ is the Weil group of $F$.
    
    These two correspondences preserve the notion of contragradient representations and hence preserve the notion of self-duality. 
    An irreducible self-dual representation $\pi$ is an orthogonal or a symplectic representation, and the parity $c(\pi)$ is defined to be $1$ if it is orthogonal and to be $-1$ otherwise.
    Hence, there arises a natural question, whether the composite of these correspondences preserves the parity or not.
   
     Prasad and Ramakrishnan \cite{Prasad} have proved a precise result below by using a globalizing method.
    
    \begin{theorem}[{\cite[Theorem A]{Prasad}}]
     \label{thmPR}
     Let $D$ be a central division algebra over $F$ of rank $d$, $G$ be $\GL_m(D)$, and $\pi$ be a discrete series representation of $G$. Let $\sigma$ denote its Langlands parameter. Then there exists a parity relation
     $$ (-1)^{md} c(\pi)= (-1)^m c(\sigma)^m .$$
    \end{theorem}    
    
    In particular, the parity inversion occurs if $G$ is the unit group of a division algebra of even rank.
    
    One can consider its conjugate self-dual analog if one can define a Galois involution on the set of isomorphism classes of irreducible representations of $\GL_m(D)$. 
    Mieda \cite{Mie} has formulated and proved a special case of this variant by constructing the involution explicitly. We state here the unramified quadratic case of his result in detail. 
    
    Let $E/F$ be an unramified quadratic extension and $D$ be a central division algebra over $E$ whose Hasse invariant is $1/n$. 
    In this case, we can write $D$ in the form $E_n[\Pi]$, where $E_n$ is the unramified extension of the $p$-adic field $E$ of degree $n$, and $\Pi$ is a prime element of $D$ such that $\Pi^n$ is equal to a uniformizer of $F$. 
    We define the automorphism $\tau$ of $D$ by declaring that it fixes $\Pi$ and it coincides with the Frobenius map over $F$ on $E_n = F_{2n}$, where $F_{2n}$ is the degree $2n$ unramified extension of $F$. 
    
    In this situation, we define that a smooth irreducible representation $(\pi, V)$ of $G$ is said to be conjugate self-dual if we have $\pi^{\tau} \cong \pi^{\vee}$ where $\pi^{\tau}(g) = \pi(\tau(g))$ and $\pi^{\vee}$ is the contragradient representation of $\pi$. 
    In addition, we have the notion of parity $c(\pi)$, defined in a similar way as the self-dual case (for the detail, see \cite[Section 2.1]{Mie}).
    
    \begin{theorem}[{\cite[Theorem 1.2]{Mie}}]
      Let $\pi$ be a conjugate self-dual representation of $D^{\times}$ and $\sigma$ be its Langlands parameter. We also assume that $\sigma$ is a supercuspidal parameter. 
      Then we have 
      $$ c(\pi) = (-1)^{n-1} c(\sigma).$$
    \end{theorem}
    
    His proof is based on $p$-adic geometry. Furthermore, he conjectured that this result could be extended to all the discrete series representations \cite[Remark 2.13]{Mie}.
    We use the global method to generalize his result. 
    One of the keys of the proof of Prasad and Ramakrishnan is the automorphic descent theorem for orthogonal and symplectic groups proved by Jiang and Soudry \cite[Appendix]{Prasad}. 
    Hence, we can consider the conjugate self-dual analog of their method if we establish the notion of self-duality in general and use the automorphic base change theorem for unitary groups by Mok \cite{Mok}. In this paper, we will show that this idea works indeed.
    
    To state our main theorem, we introduce some notations.
    Let $E/F$ be a quadratic extension of $p$-adic local fields and let $D$ be a central division algebra of rank $d$ over $E$. For simplicity, we assume that $D$ is not split.
    We write the Hasse invariant of $D$ in the form $s/d$, where $0 < s < d $ and $\gcd(s, d)=1$.
    Let $A=M_m(D)$ be a central simple algebra over $E$ and let $n=md$ be its rank over $E$. 
    Then if we take the central simple algebra $A'$ over $F$ of rank $2n$ whose Hasse invariant is $s/2d$, we have an inclusion $\phi \colon A \to A'$ over $F$. 
    Furthermore if we put $G$ = $A^{\times}$ and $G'$ = $\Norm_{A'^{\times}}(A^{\times})$, the normalizer of $A^{\times}$ in $A'^{\times}$, we have an exact sequence
          $$ 1 \to G \xrightarrow{\phi}{G'} \to \Zbb/2\Zbb \to 1 .$$     
    We have the notion of conjugate self-duality and parity $c(G, G',\pi) \in \{ \pm 1 \}$ with respect to the inner automorphism $\Int(\tau)$ induced by $\tau \in G' \setminus G$.
    The details are given in Section \ref{parity}.
    We also have the notion of conjugate self-duality and parity for inclusion of the Weil groups $ W_E \to W_F $.
    
    We now state our main theorem.
         
	\begin{theorem} [Main Theorem]
	
	\label{MainThm}
		Let $(G, G')$ be the pair associated with $A$ above.
		Let $\pi$ be a conjugate self-dual discrete series representation of $G$ = $A^{\times}$ and let $\sigma$ denote its Langlands parameter. Then we have
		$$ c(G, G',\pi) = (-1)^{(n-1)ms} c(W_E, W_F, \sigma)^{ms} .$$
	
	\end{theorem}
	
	At the moment, there seem few explicit descriptions of the local  Jacquet-Langlands correspondence for general non-cuspidal discrete series representations, as mentioned in the introduction in \cite{SecherreStevens}. Our main theorem is valid in the case that we do not have explicit descriptions.
	 
	\subsection*{Outline of the proof}
	We will follow the globalizing method used in \cite{Prasad} to prove their main theorem. Here, we give the method of our proof in a nutshell.
	
	 \begin{enumerate}
	  \item 
	  We first define the notion of conjugate self-duality which generalizes the notion defined by Mieda \cite[Section 2.10]{Mie}. 
    In the literature, the only local quadratic cases are treated.
    To apply the global method, we need to extend the notion to quadratic extensions of both local and global fields. 
    The extension will be done by means of the theory of central division algebras in Sections 1 and 2, especially in Construction \ref{locext}. 
	  
	  This construction comes from Mieda's construction in the unramified quadratic case.
    It indicates that the automorphism $\tau$ in the setting of our main theorem is derived from the inclusion of the division algebras $E_n[\Pi] = D \to D'=F_{2n}[\Pi']$. 
    Here, $D'$ is the central division algebra over $F$ of Hasse invariant $1/2n$ and $\Pi'$ is a prime element of $D'$ whose square is $\Pi$. Indeed, the automorphism $\tau$ above is the inner automorphism induced by $\Pi'$.
	  
	  \item
	  We then introduce some global theorems by Mok and Badulescu-Jacquet-Langlands in Section 3. 
    We use these theorems in a similar way as the globalizing method of Prasad and Ramakrishnan.
	  
	  \item
	  We then prove some local results in Section 4 which will be used to obtain a product formula of the parity, which is an analog of the \cite[Theorem C]{Prasad}.

    We introduce some notations in order to state our product formula.
      Let $L/K$ be a quadratic extension of number fields and $G$ denote $\mathbf{A}^{\times}(\Abb_L)$, where $\mathbf{A}$ is a central simple algebra over $L$. 
      Furthermore, let $\mathbf{B}$ be a central simple algebra over $K$ which satisfies $ \mathbf{B} \otimes_{K} L \cong M_2(\mathbf{A})$. 
      In this situation, we have an inclusion of $K$-algebras $\mathbf{A} \to \mathbf{B}$ and if we define $G'$ to be $\Norm_{\mathbf{B}^{\times}}(\mathbf{A}^{\times})(K)G$ then we have an exact sequence
      $$ 1 \to G \xrightarrow{\phi}{G'} \to \Zbb/2\Zbb \to 1 .$$
      The proof of these facts will be given in Subsection 2.1.
      Hence, we obtain the notion of self-duality again.
      
     \begin{proposition}[Proposition \ref{prodform}]
        
        Assume that $\mathbf{A}$ is split at every archimedean place of $L$.
        Let $\Pi = \bigotimes_{w} ' \Pi_{w}$  be an automorphic conjugate self-dual representation which appears in the discrete spectrum of $G$ with a unitary central character $\omega$. Let also $S$ be the set of places of $K$ where $\mathbf{B}$ does not split and put $\Pi_{v} = \bigotimes_{w|v} \Pi_w$ for each place $v$ of $K$. Then we have 
        $$ \prod_{v \in S} c(G_v, G'_v, \Pi_{v}) = 1 .$$
     \end{proposition}  
       
       By using this product formula and by controlling the places at which $\mathbf{B}$ splits, we can reduce the general case to the case already proved. The method of this reduction is the same as the one used in \cite{Prasad} and the idea is stated in the right after the statement of Theorem C $loc.\ cit$. 
        
       \item
       We finally prove the main theorem by combining the globalizing results introduced in Section 3 and the product formula above.
         
	\end{enumerate}
	
  \subsection*{Notations}
   All the underlying spaces of representations of a group are assumed to be $\Cbb$-vector spaces in this paper.
   We also assume that representations of locally compact totally disconnected groups are smooth (i.e. non-degenerate).

	\subsection*{Acknowledgments}
		This is master's thesis of the author and he is very grateful to his advisor Yoichi Mieda for his constant encouragement and earnest support. Without his help, the author could not write this paper. 
		He also thanks Junnosuke Koizumi for useful discussions.

\section{A notion of conjugate self-duality} 

        In this section, we recall a notion of conjugate self-dual representations and their parity. 
        The notion presented here is a special case of the definition in \cite[Section 2.1]{Mie}.

	\subsection{Basic definitions}

		\begin{definition}
		A pair of locally compact groups $(G, G')$ is said to be a conjugating pair if $G$ is an open subgroup of $G'$ of index $[G':G]=2$.
		\end{definition}
	    
	    \begin{example}[Inclusion of the Weil groups associated with quadratic extensions]
	    \label{Weil}
	    We take a quadratic extension of local fields $E/F$ and let $G$ be the Weil group $W_E$ of $E$ and $G'$ be the Weil group $W_F$ of $F$. 
	    Then $(W_E, W_F)$ is a conjugating pair. (See \cite[Section 2]{GGP}.) 
	    \end{example}
	    
	    \begin{example}[Component-wise conjugation]
	    \label{conj}
	    We take a quadratic extension of local fields $E/F$ and let $\conj$ be the non-trivial element of the Galois group of $E/F$. Furthermore, let $G$ be $\GL_n(E)$ and $G'$ be the semi-direct product of $G$ and $\Zbb/2\Zbb$ defined by the component-wise conjugating action induced by $\conj$.
	    Then $(G, G')$ is a conjugating pair.
	    \end{example}
	    
	    \begin{example}[Switching components]
	    \label{switch}
	    Let $E$ be a quadratic split \'etale algebra $F \times F$ over a local field $F$. 
	    Then $\GL_n(E)$ is canonically isomorphic to $\GL_n(F) \times \GL_n(F)$  by the product of two projections from $E$ to $F$. We define an action of $\Zbb/2\Zbb$ on $\GL_n(E)$ by switching those two components. This action defines a conjugating pair as in Example \ref{conj}.
	    \end{example}
	    
	    \begin{example}[Non-trivial construction over local fields]
	    See \cite[Definition 2.10]{Mie} for non-trivial constructions of conjugating pairs from division algebras over non-archimedean local fields. We have already introduced the case of unramified quadratic extensions in our introduction. 
      In Construction \ref{locext}, we generalize this construction using the theory of central simple algebras over fields.
	    \end{example}

	    We have the notion of conjugate self-dual representations for every conjugating pair of locally profinite groups.
	    Let $(G, G')$ be a conjugating pair of locally profinite groups.
	    For an element $\tau$ of $G'$, let $\Int(\tau)$ denote the inner automorphism induced by $\tau$.
	    If we take a smooth representation $(\pi, V)$ of $G$, we write $(\pi^{\tau}, V)$ for the representation defined by $\pi^{\tau}(g)=\pi(\Int({\tau})(g))$ and $\pi^{\vee}$ for the contragradient representation of $\pi$.
	    
	    \begin{definition}
	    A smooth irreducible representation $(\pi, V)$ of $G$ is said to be conjugate self-dual with respect to $(G, G')$ if there exists an element $\tau \in G' \setminus G$ such that $\pi^{\tau} \cong \pi^{\vee}$.
	
		\end{definition}
		
		\begin{remark}
		If $(\pi, V)$ is conjugate self-dual in the above sense, then $\pi^{\tau} \cong \pi^{\vee}$ for $all$ $\tau \in$ $G'\setminus G$. 
    In addition, the second dual $\pi^{\vee \vee}$ of $\pi$ is equivalent to $\pi$ because we have $\tau^2 \in G$. 
    Hence, $(\pi, V)$ is an admissible representation in this case.
		\end{remark}

	\subsection{Parity associated with a conjugating pair}
		\label{parity}
		We fix $\tau \in G'\setminus G$ for a moment, and we recall the definition of parity in \cite[Section 2.1]{Mie}.
		
		Let $(G, G')$ be a conjugating pair and $(\pi, V)$ be a conjugate self-dual representation with respect to $(G, G')$.
		Then, there exists a non-degenerate bilinear pairing $\langle \cdot, \cdot \rangle$$\colon$$V \times V \to \Cbb$ satisfying $\langle \pi(\tau(g))(x), \pi(g)(y) \rangle$ = $\langle x, y \rangle$ for all $x,y \in V$ and $g \in G$. By Schur's lemma, such a pairing is unique up to scalar multiplication.
		
		Since our setting is now a special case of \cite[Section 2.1]{Mie}, we can use the results proved there.
    For example, we have

        \begin{lemma}
        There exists $C(G,G',\pi) \in \{\pm1\}$ such that 
        $$\langle \pi(\tau^2)y, x \rangle = C(G,G',\pi) \langle x, y \rangle.$$
        It does not depend on which $\tau \in G'\setminus G$ we take.
        \end{lemma}

		\begin{proof}
		If we take $\tau^2$ as the element $t$ in \cite{Mie}, it follows from Lemmas 2.1 and 2.5 $loc. \ cit.$ 
		\end{proof}
		
		\begin{definition}
		We call the above constant $C(G,G',\pi)$ the parity of $\pi$ with respect to $(G, G')$.
    The representation $\pi$ is said to be conjugate orthogonal (resp. conjugate symplectic) if the parity is equal to $1$ (resp. $-1$).
		\end{definition}
		
		The next lemma follows trivially from the definition of parity.

		\begin{lemma}[{\cite[Lemma 2.5]{Mie}}]
		\label{multi}
    Let $(G, G')$ and $(H, H')$ be conjugating pairs and let $\pi$ and $\rho$ be conjugate self-dual representations of $G$ and $H$, respectively.
    Then, the exterior tensor product $\pi \boxtimes \rho$ is a conjugate self-dual representation of $G \times H$, and furthermore, the parity of  $\pi \boxtimes \rho$ is the product of the parity of $\pi$ and $\rho$.
		\end{lemma}

        \begin{example}
        If we take a conjugating pair as in Example \ref{Weil}, we obtain the notion of conjugate self-duality and parity defined in \cite[Section 2]{GGP}.
        Similarly, if we take such a pair as in Example \ref{conj}, we obtain the usual notion of conjugate self-duality and parity.
        \end{example}

        \begin{remark} 
        \label{realcase}
		We also treat conjugate self-dual representations of archimedean groups (strictly speaking, $(\mathfrak{g},K)$-modules) in this paper. The definition above is easily extended to the archimedean case, i.e. the case where $G$ and $G'$ are reductive Lie groups, by applying Dixmier's lemma instead of Schur's lemma.
		\end{remark}

\section{Central simple algebra settings}
   \subsection{A construction of conjugating pairs}
        
        In this section, we construct conjugating pairs from quadratic extensions of fields and central simple algebras over those fields. 
        As we noted in the introduction, we need to construct conjugating pairs in both local and global situations to apply the globalizing results. For this reason, we consider general quadratic extensions.
        
        Let $E/F$ be a field extension of finite degree $n \in \Zbb_{>0}$. 
        We do not assume that $E$ or $F$ to be local fields or global fields.
        In addition, let $A$ be a central simple algebra over $E$ and $B$ be a central simple algebra over $F$. 
        
        The special case $A=E$ of the following lemma is related to the splitting subfield of a central simple algebra.

        \begin{lemma}
        \label{keylemma1}
        Assume that we have an $E$-algebra isomorphism $ B \otimes_{F} E \cong M_n(A)$.
        Then, we have an $F$-algebra homomorphism $A \to B$.    
        \end{lemma}

        \begin{proof}
        We take a basis of $E$ over $F$ to obtain a homomorphism of $F$-algebras $E \to M_n(F)$. By composing this map to the given isomorphism, we obtain an $F$-algebra homomorphism
        $$\phi \colon A \otimes_{F} M_n(F) \cong B \otimes_{F} E \to B \otimes_{F} M_n(F) .$$
        
        We use the Skolem-Noether theorem for central simple $F$-algebras to take $c \in (B \otimes_{F} M_n(F))^{\times} $ such that 
        $$(\Int(c) \circ \phi)(F \otimes_{F} M_n(F))=F \otimes_{F} M_n(F).$$
        By considering the centralizers of $F \otimes_{F} M_n(F)$ both in $A \otimes_{F} M_n(F)$ and $B \otimes_{F} M_n(F)$, we finally get an $F$-algebra homomorphism
        $$\Int(c) \circ \phi \colon A \to B.$$       
        \end{proof}

        \begin{remark}
        By the Skolem-Noether theorem for central simple $F$-algebras, two homomorphisms as in this lemma are conjugate under an inner automorphism of $B$.
        \end{remark}

        In the following of this section, we take $E/F$ as a quadratic extension of fields (hence we take $n=2$) and we assume that the pair $(A, B)$ satisfies the assumption of Lemma \ref{keylemma1}. 
        Under this assumption, we obtain a homomorphism $\phi \colon A \to B$ over $F$, and for a moment, we fix that homomorphism.
        
        Let $\Gbb$ denote the reductive group $\Res_{E/F}(A^{\times})$ over $F$, and $\Hbb$ denote the reductive group $B^{\times}$ over $F$.

        \begin{lemma}
        There exists an exact sequence of algebraic $F$-groups
        $$ 1 \to \Gbb \xrightarrow{\phi}{\Norm_{\Hbb}{\Gbb}} \to \underline{\Zbb/2\Zbb}_{F} \to 1.$$
        Here, $\underline{\Zbb/2\Zbb}_{F}$ denotes the constant group scheme associated with $\Zbb/2\Zbb$.
        \end{lemma}

        \begin{proof}
        Let $\Qcal$ denote the quotient of the second map of sequences.
        Then we obtain the exact sequence 
        $$ 1 \to \Gbb \xrightarrow{\phi}{\Norm_{\Hbb}{\Gbb}} \to \Qcal \to 1.$$
        We consider the base change of this sequence to an algebraic closure $\overline{F}$ of $F$. 
        Then this sequence becomes 
         $$ 1 \to \GL_n({\overline{F}})\times\GL_n({\overline{F}})  \xrightarrow{\phi_{\overline{F}}}{\Norm_{\GL_{2n}({\overline{F}})}({\GL_n({\overline{F}})\times\GL_n({\overline{F}}) })} \to \Qcal_{\overline{F}} \to 1,$$
        where $n$ is the rank of $A$ over $E$.
        
        We claim that $\Qcal_{\overline{F}}$ is isomorphic to $\underline{\Zbb/2\Zbb}_{\overline{F}}.$
        By using the Skolem-Noether theorem over $\overline{F}$, we may assume that $\phi_{\overline{F}}$ is the diagonal embedding. Then the image of this embedding in $\GL_{2n}({\overline{F}})$ is a maximal Levi subgroup, the claim then follows easily. 
        
        The lemma follows from this claim and the fact that $\Qcal$ has the identity element as an $F$-valued point.
        \end{proof}
    
        Furthermore, we assume that $F$ is infinite in the following argument. Let $G$ denote $\Gbb(F)$ and also $H$ denote $\Hbb(F)$. Then, by taking the long exact sequence of Galois cohomology of the above and applying Hilbert's Theorem 90 \cite[Lemma 2.7.4]{GS}, we obtain$\colon$
        
        \begin{corollary}
        \label{conjpair}
        Assume that $F$ is infinite.
        Then there exists an exact sequence 
        $$ 1 \to G \xrightarrow{\phi}{\Norm_{H}{G}} \to {\Zbb/2\Zbb} \to 1.$$
        \end{corollary}
        
        \begin{proof}
        It suffices to check that $\Norm_{H}{G}$ = $(\Norm_{\Hbb}{\Gbb})(F)$. This equality follows from the density result in \cite[Corollary 13.3.9]{Springer}.
        \end{proof}
        
        Hence in the above setting, we can obtain the conjugating pair $(G, \Norm_{H}{G})$. 
        
        \begin{definition}
        The above conjugating pair is said to be associated with $(A, B, \phi)$.
        \end{definition}

        \begin{remark}
        By the theorem of Skolem-Noether over $F$, the notion of parity and conjugate self-duality with respect to $(G, \Norm_{H}{G})$ does not depend on how to take $\phi$. So in the following argument, we omit $\phi$.
        \end{remark}

        By the Zariski density of $G$ in $\Abb$, that is $A$ viewed as a scheme over $F$, $\Norm_{H}{G}$ = $\Norm_{H}{\Abb}$ = $\Norm_{H}{A}$.
        Hence, each element of $\Norm_{H}{G}$ induces an automorphism of $A$ and hence of its center $Z(A)=E$. So we have a map
        $$\Norm_{H}{G} \to \Gal(E/F).$$

        \begin{lemma}
        \label{galconj}
        The map above is surjective.
        \end{lemma}
              
        \begin{proof}
        Since $E$ is a simple algebra over $F$ and $B$ is a central simple algebra over $F$, we have a dimension formula $[E:F][\Cent_{B}(E):F]=[B:F]$ (see \cite[Theorem 2.43]{KnappAA}).
        As $\Cent_{B}(E)$ contains $A$, the equation above shows that they have the same dimensions over $F$.
        It follows that $\Cent_{B}(E)$ = $A$.
        Hence, all the elements of $\Norm_{H}{G} \setminus G$ map to the non-trivial element of the Galois group, and the map above is surjective.
        \end{proof}
        
        By Corollary \ref{conjpair} and Lemma \ref{galconj}, we have obtained an automorphism of $A$ that induces the non-trivial Galois involution on $Z(A)=E$.
        
        \begin{remark}
        Note that this automorphism induces a non-trivial Galois action on the set of isomorphism classes of the irreducible discrete series representations of $\GL_n(E)$ through the local Jacquet-Langlands correspondence.
        \end{remark} 
     
        \begin{remark}
        The sequence in Lemma \ref{conjpair} does not split in general. Indeed, if the sequence splits, it is necessary that $A$ has an involution that induces the non-trivial Galois involution on $Z(A)=E$. Then by the Galois descent, $A$ has to be a base change of a central simple algebra over $F$. However, that cannot be true if we take any quaternion algebra over $E$.
        \end{remark}

   \subsection{Notations in local and global settings}
     \label{notation}   
        In this subsection, we fix notations for pairs of central simple algebras over both local fields and number fields used in our globalizing method.
        
        In the following of this paper, we only treat number fields and $p$-adic and archimedean local fields, so all the fields below are assumed to be of characteristic zero unless otherwise stated.
  
        First, we treat the case of extensions of local fields. 
        
        \begin{construction}
        \label{locext} 
         Let $E/F$ be a quadratic extension of local fields. 
         In addition, let $A=M_m(D)$ be a central simple algebra of rank $n=md$ over $E$.
         If the rank d of $D$ is strictly larger than one, we write $s/d$ for the Hasse invariant of $A$ in the way that $s$ satisfies $0 < s < d$ and $\gcd(s, d)=1$. 
         Let $B$ be the central simple algebra over $F$ of rank $2n$ whose Hasse invariant is $s/2d$. If $d = 1$, we take $s = 0$ and let $B$ be the central simple algebra $M_{2n}(F)$ over $F$.
        
        Then, we have $B \otimes_{F} E \cong M_2(A)$ and obtain the conjugating pair $(G, G')$ associated with $(A, B)$ using the results in the previous subsection.
        \end{construction}

        Next, we treat the case of quadratic split \'etale algebras.
        Let $E$ be a quadratic split \'etale algebra $F \times F$ over $F$. 
 
        \begin{construction}
        \label{locspl}
        We associate the conjugating pair of Example \ref{switch}.
        \end{construction}
   
        Finally, we consider a conjugating pair in a global setting.
        We only use the following special constructions.
        
        Let $L/K$ be a quadratic extension of number fields and $v$ be a finite place of $K$ that does not split in this extension. 
        Let $L_{v}$ denote the completion of $L$ with respect to the place $v$ of $L$.
        We assume that we are given a central simple algebra $A$ over $L_v$ of rank $n=md$ over $L_v$. We write $s/d$ for the Hasse invariant of $A$ as in Construction \ref{locext}.  
        
        We also assume that we are given $ms$ finite places $v_1, \ldots v_{ms}$ that also do not split and unramified in $L$, and central division algebras $A_i$ over $L_{v_i}$ whose Hasse invariant is equal to $-1/n$.
        
        \begin{construction}
        \label{global}
        By the theorem of Brauer-Hasse-Noether (\cite[p.184 Theorem]{Mum}, \cite[XIII \S 6, Theorem 4]{Wei}), we can globalize $(A, \{A_i\})$ uniquely to a central division algebra $\mathbf{D}$ over $L$ by declaring that we take split central simple algebras at other places.
        
        Moreover, by the similar construction taking $B$ of Construction \ref{locext} at $v$ and taking central division algebras $B_i$ over $K_{v_i}$ whose Hasse invariant is equal to $-1/2n$ for $i = 1, 2, \ldots ,{ms}$, we obtain a central division algebra $\mathbf{E}$ over $K$.
        
        Then we have $\mathbf{E} \otimes_{K} L \cong M_2(\mathbf{D})$, so we associate the conjugating pair $(G, G')$, where $G = \mathbf{D}^{\times}(\Abb_{L})$ and $G' = \Norm_{\mathbf{E}^{\times}}(\mathbf{D}^{\times})(K) G$, with $L/K$.
        \end{construction}
        
        \begin{remark}
        \label{globalaux}
        In general, if we have a quadratic extension of number fields $L/K$ and central simple algebras $\mathbf{A}$ over $L$ and $\mathbf{B}$ over $K$, which satisfy $\mathbf{B} \otimes_{K} L \cong M_2(\mathbf{A})$, we can do the same construction as above.
        \end{remark}
        
        Unless otherwise stated, we will use these notations of conjugating pairs and only consider the conjugate self-duality with respect to them.

\section{Global results what we use}

   \subsection{A globalizing result}
        
        We review a globalizing result for conjugate self-dual discrete series representations due to Mok \cite{Mok}.
        
        \begin{lemma}
         \label{globalizefld}
        Let $E/F$ be a quadratic extension of $p$-adic fields.
        There exists a quadratic extension $L/K$ of number fields such that $K$ is totally real, $L$ is totally imaginary, and also it has a finite place $v$ of $K$ that does not split in $L$ and the completion of $L/K$ at $v$ is $E/F$.

        \end{lemma}
   
        \begin{proof}
         This is an easy consequence of Krasner's lemma.
        \end{proof}

      Following \cite[Sections 2.1 and 2.4]{Mok}, we introduce some notations about the base change map for the quasi-split unitary group.
      
      We first define Langlands parameters for reductive groups following \cite{BorelLfunctions} and \cite{AubertLL}.
      
      \begin{definition}
       \label{Lgroup}
       Let $F$ be a non-archimedean local field.
       We define $L_F$ as the direct product $W_F \times \SL_2(\Cbb)$, where $W_F$ be a Weil group of $F$.
       Let $G$ be a connected reductive group over $F$. We write ${}^LG$ for the $L$-group of $G$. 
       It is defined as follows$\colon$ Let $\Phi$ = $(X, \Delta, X^{\vee}, \Delta^{\vee})$ be the based root datum of $G$ and $\Phi^{\vee}$ = $(X^{\vee}, \Delta^{\vee}, X, \Delta)$ be its dual. 
       There exists the reductive group $G^{\vee}$ over $\Cbb$ whose root datum is equal to $\Phi^{\vee}$. 
       It is defined up to isomorphism and called the dual group of $G$. 
       If we fix a splitting of $G^{\vee}$, the action of $W_F$ on $\Phi$ induces an action of $W_{F}$ on $G^{\vee}$. 
       The semi-direct product defined by this action is the $L$-group of $G$.
       \end{definition}
      
       \begin{definition}
        \label{Lparameter}
        A Langlands parameter for ${}^LG$ is a continuous homomorphism 
         $$ \phi \colon L_F \to {}^LG $$
        such that
        \begin{itemize}
         \item $\phi$ commutes with the projections to $W_F$. 
         \item $\phi(w)$ is semisimple for all $w \in W_F$.
         \item $\phi_{ | \SL_2(\Cbb)}$ is algebraic.
        \end{itemize}
      
        We define $\Phi({}^LG)$ as the set of $G^{\vee}$-conjugacy classes of Langlands parameters for ${}^LG$.
       \end{definition}

       \begin{definition}
       A parabolic subgroup $P^{\vee}$ of $G^{\vee}$ is $F$-relevant if the set of  simple roots attached to $P^{\vee}$ corresponds to a parabolic subgroup of $G$ defined over $F$. 
       A parabolic subgroup $P^{\vee}$ of $G^{\vee}$ is said to be $W_F$-quasi-stable if $\Norm_{{}^LG}(P^{\vee}) \to W_F$ is surjective.
       
       Let $\phi$ be a Langlands parameter for ${}^LG$. 
       Let $P^{\vee}$ be a $W_F$-quasi-stable parabolic subgroup with Levi factor $L^{\vee}$ such that Im$(\phi)$ is contained in $\Norm_{{}^LP}(L^{\vee})$, where ${}^LP$ = $W_F \ltimes P^{\vee}$, and $P^{\vee}$ is minimal with this property. 
       The parameter $\phi$ is said to be $G$-relevant if $P^{\vee}$ is $F$-relevant. 
       
       Let $\Phi(G)$ be the subset of $G$-relevant Langlands parameters for $\Phi({}^LG)$.
       \end{definition}
      
       \begin{definition}
       A Langlands parameter $\phi$ for ${}^LG$ is said to be discrete if there exists no proper $W_F$-stable Levi subgroup $L^{\vee}$ such that ${}^LL$ contains $\phi(L_F)$. It is said to be bounded if the projection of $\phi(W_F)$ to $G^{\vee}$ is bounded. 
       \end{definition}
      
       \begin{remark}
        \label{dpchar}
        The discreteness of a parameter $\phi$ is equivalent to finiteness of $\Cent_{G^{\vee}}(\Im{\phi})/Z(G^{\vee})^{W_F}$, where $\Cent_{G^{\vee}}(\Im{\phi})$ is the centralizer of $\Im{\phi}$ in $G^{\vee}$. See the remark before \cite[Lemma 3.1]{GrossReeder} and the proof of \cite[Proposition 3.6]{BorelLfunctions}.
       \end{remark}

       Let $E/F$ be a local or global quadratic extension and $C_E$ be $E^{\times}$ in the local case and $\mathbb{A}^{\times}_E / E^{\times}$ in the global case.
       Let $\mathcal{Z}_E$ be the set of unitary characters of $C_E$ which are conjugate self-dual with respect to the Galois action. 
      We define $\mathcal{Z}_E^{+1}$ (resp. $\mathcal{Z}_E^{-1}$) as the set of elements of $\mathcal{Z}_E$ whose restriction to $C_F$ is trivial (resp. the quadratic character corresponding to $E/F$). It is easily seen that these sets $\mathcal{Z}_E^{+1}$ and $\mathcal{Z}_E^{-1}$ are non-empty as follows$\colon$

      \begin{lemma}
        Let $E/F$ be a quadratic extension of number fields.
        Then, the sets $\mathcal{Z}_E^{+1}$ and $\mathcal{Z}_E^{-1}$ are non-empty. 
      \end{lemma}

      \begin{proof}
        The statement about $\mathcal{Z}_E^{+1}$ is trivial because the trivial character belongs to $\mathcal{Z}_E^{+1}$.
        Hence, it suffices to show the claim that we can extend the quadratic character of $\mathbb{A}^{\times}_F / F^{\times}$ to $\mathbb{A}^{\times}_E / E^{\times}$.
        If we can prove that $\mathbb{A}^{\times}_E / E^{\times}\Nm_{E/F}({\Abb^{\times}_{E}})$ is locally compact Hausdorff, then the claim follows from \cite[Theorem 4.39]{FollandAHA}.
        There is a canonical homeomorphism
        $$ \mathbb{A}^{1}_E / E^{\times}\Nm_{E/F}({\Abb^{1}_{E}}) \to \mathbb{A}^{\times}_E / E^{\times}\Nm_{E/F}({\Abb^{\times}_{E}})  $$
        induced from a canonical inclusion 
        $$ \mathbb{A}^{1}_E / E^{\times} \to \mathbb{A}^{\times}_E / E^{\times} .$$
        Hence, it suffices to show that $\mathbb{A}^{1}_E / E^{\times}\Nm_{E/F}{\Abb^{1}_{E}}$ is Hausdorff. 
        However, this follows from the fact that $\mathbb{A}^{1}_E / E^{\times}$ is Hausdorff and $\Nm_{E/F}({\Abb^{1}_{E}})E^{\times} / E^{\times} = \Im(\Nm_{E/F}\colon \mathbb{A}^{1}_E / E^{\times} \to \mathbb{A}^{1}_E / E^{\times})$ is compact.
      \end{proof}

      Furthermore, let $\mathrm{U}_{E/F}(n)$ be the quasi-split unitary group of $n$-variables for a quadratic extension $E/F$.
       
       \begin{definition}[{\cite[Section 2.1]{Mok}}]
       Let $\kappa$ belong to $\{ \pm 1 \}$ and $E/F$ be a local quadratic extension.
       Furthermore let $\chi_{\kappa}$ belong to $\Zcal_E^{\kappa}$.
       We define the base change map $\xi_{\chi_{\kappa}} \colon \Phi(\mathrm{U}_{E/F}(n)(F)) \to \Phi({GL_n(E)})$ as 
       $$ \phi \mapsto ((pr_1 \circ \phi_{|L_E}) \otimes \chi_{\kappa}, pr_2 \circ \phi_{|L_E}),$$
       where $pr_1$, $pr_2$ are the projections from ${}^LU_{E/F}(n)$ to $U_{E/F}(n)^{\vee}$ and $W_F$, respectively, and we identify an element of ${\Zcal}_E^{\kappa}$ with a character of the Weil group of $E$ via the local class field theory.  

       \end{definition}
       
       \begin{remark}
       The map is induced from an $L$-embedding $ {}^L{\mathrm{U}(n)} \to {}^{L}{\Res_{E/F}(GL_n)}$. See \cite[(2.1.9)]{Mok}.
       \end{remark}
       
       \begin{lemma}
       \label{bc}
       \begin{enumerate}
       \item The base change map is injective and its image equals the set of conjugate self-dual parameters with parity $(-1)^{n-1} \kappa$.
       \item Let $\phi$ be an element of $\Phi(\mathrm{U}_{E/F}(n))$ and $\phi' = \xi_{\chi_{\kappa}}(\phi)$ be its base change. If $\phi'$ is a bounded and discrete parameter, then $\phi$ is also bounded and discrete.
       \end{enumerate}
       \end{lemma}
   
       \begin{proof}
       \begin{enumerate}
       \item 
       It is \cite[Lemma 2.2.1]{Mok}.
       \item It easily follows from the criterion of Remark \ref{dpchar} and Schur's lemma.  
       \end{enumerate}
       \end{proof}
       
       The local Langlands classification for tempered representations of the quasi-split unitary group is proved by Mok; see \cite[Theorem 2.5.1]{Mok}. 
       Roughly speaking, for each bounded Langlands parameter $\phi$, there exists a finite set of tempered irreducible representations $\Pi_{\phi}$ called an $L$-packet associated to $\phi$.
       We will use the following consequence of the local classification theorem by Mok.
       
       \begin{theorem}[{\cite[Section 7.7, (7.7.17)]{Mok}}]
        \label{discpar}
       Let $E/F$ be a local quadratic extension. 
       If a Langlands parameter $\phi$ for $\mathrm{U}_{E/F}(n)$ is discrete and bounded, the associated $L$-packet $\Pi_{\phi}$ consists of square-integrable representations.  
       \end{theorem}
       
       \begin{proof}
       It follows from the construction of $L$-packets in \cite[Section 7.7]{Mok} and especially remarks after \cite[Proposition 7.7.4]{Mok}.
       \end{proof}
       
       \begin{theorem}
        \label{globalizeunit}
       Let $L/K$ be a quadratic extension of number fields such that $K$ is totally real and $L$ is totally imaginary.
       Let also $v_1, \ldots,v_m$ be a set of finite places which does not split in $L/K$ and we fix a place $v$ of $K$ which splits in $L/K$.
       Furthermore, let $\pi_1, \ldots, \pi_m$ be discrete series representations of the groups $\mathrm{U}_{L_{v_i}/K_{v_i}}(n)(K_{v_i})$. 
       Then there exists a cuspidal automorphic representation $\Pi$ of $\mathrm{U}_{L/K}(n)(\Abb_K)$ whose local component at each $v_i$ is isomorphic to $\pi_i$ and supercuspidal at $v$.
       \end{theorem}
       
       \begin{proof}
       It follows directly from \cite[Corollary 4.5]{ShinPlancherel} by taking $S$ as the set ${v_1, \ldots, v_m}$ and $U$ as the point corresponding to discrete series representations and taking the restriction of a matrix coefficient of a supercuspidal representation to ${}^0G$ (see \cite[V.2.3]{Renardp-adique}) as the test function at $v$.
       We can replace it with a modification of the proof of \cite[Theorem 1B]{Clozellimit}.
       \end{proof}
       
       We also use a corollary of the global classification theorem by Mok \cite{Mok}.
       We recall some notations to state it.

       \begin{definition}[{\cite[Section 2.3]{Mok}}]
         Let $E/F$ be a quadratic extension of number fields.
         \begin{enumerate}
           \item We define a simple global parameter of $\GL_{N}(\Abb_{E})$ to be a formal tensor product
                 $$ \mu \boxtimes \nu $$
                 where $\mu$ is a unitary cuspidal automorphic representation of $GL_m(\Abb_E)$ and $\nu$ be an $n$-dimensional algebraic representation of $\SL_2(\Cbb)$, such that $mn=N$.
           \item A finite formal direct sum 
                 $$ l_1(\mu_{1}\boxtimes \nu_{1}) \boxplus \cdots \boxplus l_r(\mu_{r} \boxtimes \nu_{r}) $$
                 is called a global parameter of $GL_N(\Abb_E)$ if each $l_i$ is a positive integer and each formal summand is a simple global parameter of $GL_{m_i}(\Abb_E)$, such that $\sum_{i=1}^r l_{i}m_{i}=N$.

           \item A global parameter is said to be elliptic if all the numbers $l_i$ above are equal to 1 and all the simple global parameters $\mu_i$ are conjugate self-dual.
         \end{enumerate}
       \end{definition}

       \begin{remark}[{\cite[Remarks after (2.3.3)]{Mok}}]
         There exists a bijection between the set of global parameters of $GL_N(\Abb_E)$ and the full automorphic spectrum of $GL_N(\Abb_E)$. A simple parameter corresponds to a discrete automorphic representation.
       \end{remark}
        
       We now state a corollary of the global classification theorem by Mok \cite{Mok}.
      
       \begin{theorem}[{\cite[Theorem 2.5.2]{Mok}}]
         \label{globalclass}
         Let $E/F$ be a quadratic extension of number fields and $\chi_{\kappa}$ be a character in $\mathcal{Z}^{\kappa}_E$.
         Then, there is a map 
         $$ \Pi \mapsto \mathrm{BC}(\Pi) $$
         from the set of representations in the discrete automorphic spectrum of $\mathrm{U}_{E/F}(N)$ to the set of elliptic global parameters of $GL_N(\Abb_E)$.
         If $\Pi$ is supercuspidal at one place $v$ of $F$ which splits in $E/F$, it induces the local base change map associated to $\chi_{\kappa}$ at all the places and also $BC(\Pi)$ is conjugate self-dual and cuspidal.
       \end{theorem}
       
       \begin{remark}
         See \cite[Definition 2.4.5, 2.4.7]{Mok} for the definition of discrete global parameter for the quasi-split unitary group, and also see \cite[Theorem 2.5.2]{Mok} for more precise statement.
       \end{remark}

       \begin{proof}
         The first assertion is in the statement of \cite[Theorem 2.5.2]{Mok}. 
         We prove the second assertion. 
         We use the notation of \cite[Chapter 2]{Mok}. 
         It suffices to show that the global parameter $\psi$ associated to $\Pi$ is generic and hence the localization of global parameter at each place associated with $\Pi$ is a Langlands parameter (not just an $A$-parameter).
         It follows from our hypothesis about supercuspidality at $v$ and the definition of localization of global parameters at a split place in \cite[p.~24]{Mok} that the $A$-parameter $\psi_v$ attached to $\Pi$ at $v$ is generic, because the associated Langlands parameter $\phi_{\psi_v}$ corresponds to a unitary supercuspidal representation of $GL_n(E_v)$.
         Hence, the global parameter $\psi$ associated to $\Pi$ is also generic and $BC(\Pi)$ is cuspidal.
       \end{proof}

       \begin{theorem}
        \label{globalizerep}
        Let $L/K$ be a global quadratic extension and $v_1,\dots,v_m$ be a finite set of finite places which do not split in this extension.
        Furthermore, let $\pi_1, \ldots,\pi_m$ be discrete conjugate self-dual representations of $\GL_{n}(L_{v_i})$ such that their Langlands parameters have the same parity.
        
        Then, there exists a conjugate self-dual cuspidal automorphic representation $\Pi$ of $\GL_n(\Abb_{L})$ whose local component at each $v_i$ is isomorphic to $\pi_i$.
       \end{theorem}
       
       \begin{proof}
        First, we fix an element $\chi$ of $\mathcal{Z}^{\kappa}_{E}$ such that $(-1)^{n-1}\kappa$ is equal to the parity of the Langlands parameter of $\pi_1$.
        We use the local components of this character $\chi$ to define the local base change map.
        By using Lemma \ref{bc} (i), we can attach the given representations $\pi_i$ to discrete bounded Langlands parameters $\phi_i$ of the quasi-split unitary group. 
        Here, we use the hypothesis that the parameters have the same parity.
       
        We then choose representations $\pi_i'$ from the $L$-packets associated with $\phi_i$. 
        By Theorem \ref{discpar}, they are discrete series representations.
        We can now globalize the representations $\pi_i'$ to cuspidal automorphic representation $\Pi'$ of $\mathrm{U}_{L/K}$ by the method of Theorem \ref{globalizeunit}.
       
        By applying Theorem \ref{globalclass}, we obtain the cuspidal conjugate self-dual automorphic representation $\Pi$ = $\mathrm{BC}(\Pi')$ of $\GL_n(\Abb_L)$ whose each localization at $v_i$ is isomorphic to $\pi_i$.
        Hence the theorem.
       \end{proof}
       
       \begin{remark}
         Takanashi and Wakatsuki give a different proof of the globalization result in Theorem~\ref{globalizerep} using the twisted trace formula \cite[Theorem 1.6(i)]{TW}.
       \end{remark}

   \subsection{Special cases of the global Jacquet-Langlands correspondence}
   
       We also use the same form of the Badulescu-Jacquet-Langlands correspondence \cite{Bad08} which have been used in \cite[Theorem 1.2]{Prasad}.

       \begin{definition}[{\cite[Definition 1.5]{Aub}}]
         Let $G$ be a connected reductive group over a $p$-adic field $F$. 
         The Aubert-Zelevinsky involution of a smooth representation $\pi$ of $G$ of finite length is defined to be 
         $$ D_{G}(\pi) = \sum_{P \subseteq G \colon \mathrm{standard \  parabolic}} (-1)^{\dim(A_P)}i_P^G(r_P^G(\pi)),$$
         where $i_P^G(\pi)$, $r_P^G(\pi)$ are the normalized induction and the normalized Jacquet module of $\pi$ with respect to $P$, and $A_P$ is the split component of the center of a Levi component of $P$. 
         It is an element of the Grothendieck group of smooth representations of $G$ of finite length (for example, see \cite[VI.6.2, VI.6.4]{Renardp-adique}).
       \end{definition}

       \begin{theorem}[{\cite[Corollaire 3.9]{Aub}}]
         If $\pi$ is irreducible, then there exists a sign $\epsilon$ such that $\epsilon D_{G}(\pi)$ is an irreducible representation.
         We denote this representation by $|D_G(\pi)|$.  
       \end{theorem}

       \begin{remark}
         The map $\pi \mapsto |D_G(\pi)|$ is said to be induced from the Aubert-Zelevinsky involution in this paper.
       \end{remark}

       \begin{theorem} 
       \label{BJL}
        \begin{enumerate}
         \item
         Let $K$ be a number field and let $\Pi$ be an automorphic representation of $\GL_n(\Abb_K)$ of unitary central character $\omega \colon \Abb_K^{\times}/K^{\times} \to \Cbb^{\times}$ which occurs in the discrete spectrum.
         Suppose $\mathbf{B}$ is a central simple algebra over $K$ of dimension $n^2$.
       
         Let $S$ be the set of places where $\mathbf{B}$ is not split.
         Assume that $S$ contains only non-archimedean places and that $\Pi_v$ is a discrete series representation for every $v\in S$.
       
         Then $\Pi$ transfers to a discrete automorphic representation $\Pi'$ of $\mathbf{B}^{\times}(\Abb_K)$.
         For each $v\in S$, let $G'_v=(\mathbf{B}\otimes_K K_v)^\times$.
         The local Jacquet-Langlands transfer of $\Pi_v$ to $G'_v$ is isomorphic to $\Pi'_v$ or to $|D_{G'_v}(\Pi'_v)|$.
       
         We also have $\Pi_{v} = \Pi'_{v}$ for $v$ which $\mathbf{B}$ splits.
       
        \item  The multiplicity one and strong multiplicity one theorems hold for $\mathbf{B}^{\times}$.
       
        \end{enumerate}
       \end{theorem}
       
       \begin{proof}
         \begin{enumerate}
          \item The local Jacquet-Langlands character identity shows that each $\Pi_v$ for $v\in S$ is compatible with $G'_v$.
          The assertion follows from \cite[Proposition 3.7 and Theorem 5.1(a)]{Bad08}.
          \item This is (b), (c) of \cite[Theorem 5.1]{Bad08}.
         \end{enumerate}
       \end{proof}

\section{Local results what we use}

       In this section, following the globalizing method of \cite{Prasad}, we prove some local results which will be combined with the global results above. 

   \subsection{Triviality of parity at almost all places}
       First, we prove that the local components of a global conjugate self-dual representation are conjugate orthogonal in almost all places.
       
       \begin{lemma}
       \label{unramified}
       Let $G$ be a totally disconnected locally compact Hausdorff topological group and $K$ be an open compact subgroup of $G$. We also assume that there exists an involution $\theta$ of $G$ which stabilizes $K$.
       Let also $\pi$ be an irreducible representation that is conjugate self-dual with respect to $\theta$ and satisfies $\dim_{\Cbb}\pi^{K} = 1$.
       Then $\pi$ is conjugate orthogonal with respect to $\theta$.
       \end{lemma}
   
       \begin{proof}
       Let $G'$ be a semi-direct product of $G$ and $\Zbb/2\Zbb$ defined by $\theta$. Then, we have a conjugating pair $(G, G')$ and use the notation in Subsection \ref{parity}. 
       Note that in this case, we have $\tau^2 = 1$.
       Let $\langle \cdot, \cdot\rangle$ be the invariant form in Subsection \ref{parity}. 
       Then the restriction of this form to $\pi^{K}$ is non-degenerate and symmetric, hence by the definition of the parity, we have $C(G, G', \pi)=1$.
       \end{proof}
       
        \begin{lemma}
         \label{ungrp}
        Let $\mathrm{U}(n)$ be the closed subgroup of $\GL_n(\Cbb)$ consisting of unitary matrices. 
        Then all the unitary finite-dimensional representations of $\mathrm{U}(n)$ are conjugate self-dual representations with respect to component-wise complex conjugation $\conj$ on $\mathrm{U}(n)$.
        \end{lemma}
       
       \begin{proof}
       Let $\delta$ be a unitary finite-dimensional representation and $\chi_{\delta}$ be its character.
       Then we have 
       $\chi_{\delta}(\conj(k)^{-1})$ = $\chi_{\delta}(^{t}k)$ and this is equal to $\chi_{\delta}(k)$ by the conjugation theorem for connected compact Lie groups.
       \end{proof}
       
       We use the notations of Subsection \ref{notation}.
       
       \begin{lemma}
        \label{split}
        \begin{enumerate}
        \item Let $E/F$ be a quadratic extension of $p$-adic fields. 
        We take the conjugating pair $(G, G')$  associated with $(M_n(E), M_{2n}(F))$.
        Then all the conjugate self-dual representations of $G$ are orthogonal.
        
        \item Let $\Cbb/\Rbb$ be the quadratic extension of archimedean local fields.
        Let the conjugating pair $(G, G')$ be associated with $(M_n(\Cbb), M_{2n}(\Rbb))$ and also let $K$ = $\mathrm{U}(n)$. 
        Then all the conjugate self-dual $(\mathfrak{g},K)$-modules of $G$ are orthogonal.
        
        \item Let $E$ be a quadratic split \'etale algebra over a local field $F$. Then all the conjugate self-dual representations of $G=\GL_n(E)$ are orthogonal.
        \end{enumerate}
       
        \end{lemma}
       
        \begin{proof}
       
         \begin{enumerate}
         \item 
         This follows from Lemma 4.1 above and \cite[Theorem 1.1]{Atobe} for $E$.
         
         \item 
         We prove this in a similar fashion as Lemma 4.1. Note that $K$ is stable under the complex conjugation.
         We use the same notation as Lemma 4.1. (See also Remark \ref{realcase})
         
         Let $(\pi, V)$ be a conjugate self-dual $(\mathfrak{g}, K)$-module of $G$.
         By \cite[Theorem 4.9]{Vog86}, there exists a $K$-type $\delta$ that appears in $\pi$ with multiplicity one. Let $V({\delta})$ be the $\delta$-isotypic component of $\pi$. 
         
         Then, it suffices to show that the restriction of a non-degenerate invariant form in Subsection \ref{parity} to $V({\delta})$ is non-degenerate.
         This follows from Lemma \ref{ungrp}.
        
         \item 
         Let $\pi$ = $\pi_1 \otimes \pi_2$ be an irreducible conjugate self-dual representation of $G$. In this case, we have $\pi_2 \cong \pi_1^{\vee}$.
         We identify these representations.
         
         Then we can explicitly give a symmetric invariant form by $\langle x_1 \otimes x_2, y_1 \otimes y_2 \rangle$ = $\langle y_2, x_1 \rangle$ $\langle x_2, y_1 \rangle$ where the pairing of right hand side is the  canonical pairing between $\pi_1$ and $\pi_1^{\vee}$.
          \end{enumerate}
          \end{proof}

   \subsection{A special case of the main theorem}
   
     In this subsection, we introduce a special case of the main theorem that has already been proven by Mieda and that we will use later.
   
         \begin{theorem}[{\cite[Theorem 1.2 and Proposition 4.5]{Mie}}]
         \label{mie}
         
         \begin{enumerate}
          \item
          The main theorem \ref{MainThm} is valid for the case where $E/F$ is at worst tamely ramified and the Hasse invariant of $D$ is $1/n$, and the representation is supercuspidal.
          \item
          In the situation of (i) and if $E/F$ is unramified, there exist both conjugate orthogonal and conjugate symplectic simple supercuspidal representations.
         \end{enumerate}
         
         \end{theorem}
   
     For references about simple supercuspidal representations, see the beginning of \cite[Section 4.1]{Mie}. Mieda computed the parity of those representations using an explicit description of the local Jacquet-Langlands correspondence for simple supercuspidal representations proved in \cite{IT}. 

     \begin{remark}
       In our conjugate self-dual case, we cannot utilize the supercuspidal representations of depth zero as in \cite{Prasad}, because there is no such representation if the rank of the division algebra is even. 
       We can deduce this fact from the classification of such representations stated in \cite[Section 4]{Prasad}.
     \end{remark}
   
   \subsection{Parity invariance under the Aubert-Zelevinsky involution}
     
     We will use Theorem \ref{BJL}, but in this correspondence, a discrete series representation $\pi$ at a non-split place might be switched its place with $|D_G(\pi)|$. 
     In this subsection, we prove that the map induced from the Aubert-Zelevinsky involution does not change the parity for conjugate self-dual representations of the group of units of central simple algebras over non-archimedean local fields, following the method of Prasad-Ramakrishnan \cite{Prasad}. 
    
     We use the notation fixed in Construction \ref{locext} and only consider the cases of non-split places.
     Note firstly that in this situation, we may assume (and do assume) that the conjugate action of a conjugating pair preserves all the parabolic and Levi subgroups over $F$ by the Skolem-Noether theorem. 
     We give the detail.

     \begin{remark}

       \begin{enumerate}
        \item Let $E/F$ be a quadratic extension of $p$-adic fields.
        We take the central simple algebra $A$ over $E$ as in Construction \ref{locext} and use the notation in the construction.
        Then, we write $A$ in the form that $A$ = $M_m(D)$ where $D$ is a division algebra over $E$ whose Hasse invariant is $s/d$.
        In this case, $B$ can be written in the form that $M_m(B')$ where $B'$ is a central simple algebra over $F$ whose Hasse invariant is $s/2d$.
        
        We obtain an $F$-algebra homomorphism $D \to B'$ by Lemma \ref{keylemma1}, and $M_m(D)\to M_{m}(B')$ by applying $M_m( \cdot )$.
        We take this homomorphism as $A \to B$.
        So, if we take an element $h \in B'^{\times} \setminus D^{\times}$ which normalizes $D^{\times}$, we can take $ \tau = \diag(h, \ldots, h) $
        as an element of $G' \setminus G$, where $\diag(\ldots)$ denotes an element of the diagonal subgroup $(B'^{\times})^n \leq \GL_{m}(B')$.

        \item It is easily seen that $\Int(\tau)$ above fixes a Haar measure of $G$, so $G'$ is unimodular.
        Also, $\Int(\tau)$ preserves all the parabolic subgroups of $G$ over $E$ and their Haar measures.
      \end{enumerate}
     \end{remark}

     \begin{definition}
       \begin{enumerate}
         \item For each parabolic subgroup $P$ of $G$ over $F$, we define $P'$ as $\Norm_{G'}(P)$.
         It is equal to the semidirect product defined by the action of $\tau$ on $P$.
         We define $M'$ as the group generated by $\tau$ and $M$ for all the Levi subgroups $M$ of $G$ over $F$.
         \item We define the normalized induction functor $i_{P'}^{G'}$ and the normalized Jacquet module functor $r_{P'}^{G'}$ for a parabolic subgroup $P$ of $G$ over $F$ in the similar fashion as the case of $G$.
       \end{enumerate}
     \end{definition}

     \begin{remark}
       \label{chain}
       If $\rho$ is a representation of Levi subgroup $M$ of a parabolic subgroup $P$ of $G$ over $F$, we easily obtain a natural isomorphism $i_{P'}^{G'}(\mathrm{ind}_M^{M'}(\rho)) \cong \mathrm{ind}_{G}^{G'}(i_{P}^{G}(\rho))$.
       We have the similar isomorphism for the Jacquet functors.
     \end{remark}
     
     \begin{proposition}
     \label{AZ}
     In the situation of Construction \ref{locext}, the map induced from the Aubert-Zelevinsky involution does not change the parity of a unitary conjugate self-dual representation $\pi$ of $G$ such that $|D_G(\pi)|$ is also unitary.
    \end{proposition}
   
       We prove this proposition by using some lemmas proved below.
       The following lemma is essentially the same statement as \cite[Lemma 3.5]{GGP}.

        \begin{lemma}
          \label{AZind}
         For the conjugating pair $(G, G')$ and a conjugate self-dual representation $(\pi, V)$, $\ind^{G'}_{G}(\pi)$ is a multiplicity-free, has length at most 2 and self-dual representation of $G'$ with the same parity as $(\pi, V)$. 
         It is unitary if $\pi$ is unitary.
        \end{lemma}
        
        \begin{proof}
         The statement about multiplicity-freeness and the length follows from the Frobenius reciprocity. The second and third assertions are obvious by the definition of conjugate self-duality and parity.
        \end{proof}
        
        \begin{definition}
          A representation of a group $G$ is said to be defined over $\Rbb$ if it has a real form.
        \end{definition}

        \begin{lemma}
          \label{AZorth}
        Let $\pi'$ be an admissible unitary multiplicity-free self-dual representation of $G'$ with parity $c$ $\in$ $\{\pm 1\}$. Then, $c$ = $1$ if and only if $\pi'$ is a representation of $G'$ defined over $\Rbb$.
        \end{lemma}
         
        \begin{proof}
        The proof of \cite[Lemma 2.2]{Prasad} works similarly in this case.
        \end{proof} 
        
        \begin{lemma}
          \label{AZdefR}
        The composite of the normalized Jacquet functor and the normalized induction for a parabolic subgroup $P$ of $G$ over $F$ induces an endomorphism of the Grothendieck group of the category of representations of $G'$ of finite length and defined over $\Rbb$.
        \end{lemma}

        \begin{proof}
          It suffices to show that the normalized Jacquet functor and the normalized induction preserve the finiteness of length.
          Considering the Clifford theory (for example, see \cite[2.9 Lemma]{BZGLn}) and the Frobenius reciprocity, it suffices to show the finiteness for representations of the form $\mathrm{ind}_G^{G'}(\rho)$, where $\rho$ is an irreducible representation of $G$.
          In this case, we have $i_{P'}^{G'}(r_{P'}^{G'}(\mathrm{ind}_G^{G'}(\rho))) \cong \mathrm{ind}_G^{G'}(i_P^{G}(r_P^{G}(\rho)))$.
          Hence, it follows from \cite[VI.6.2, VI.6.4]{Renardp-adique} and the Clifford theory above.
        \end{proof}
        
        \begin{lemma}
          \label{ssdefR}

          Let $G$ be a totally disconnected group.
          A semisimple representation $\pi$ of finite length of $G$ is defined over $\Rbb$ if its image in the Grothendieck group of representations of $G$ of finite length lies in the subgroup generated by representations of $G$ of finite length and defined over $\Rbb$. 
         \end{lemma}

         \begin{proof}
           By taking composition series of the real forms and using exactness of complexification, we write $[\pi]=\sum_W a_W[W_{\Cbb}]$, where the $W$ are pairwise nonisomorphic irreducible real representations, $W_{\Cbb}=W\otimes_{\Rbb}\Cbb$, and $a_W\in\Zbb$.
           Let $J$ be the canonical conjugation on $W_{\Cbb}$.
           The underlying real representation of $W_{\Cbb}$ is $W\oplus W$, so $W_{\Cbb}$ has finite length and contains an irreducible complex subrepresentation $U$.
           The sum $U+J(U)$ is the complexification of a nonzero real subrepresentation of $W$.
           Thus $U+J(U)=W_{\Cbb}$, and $W_{\Cbb}$ is semisimple.

           The underlying real representation of every irreducible constituent of $W_{\Cbb}$ is $W$-isotypic, since it is a subrepresentation of $W\oplus W$ over $\Rbb$.
           Hence distinct $W$ give disjoint sets of irreducible complex constituents.
           The nonnegative multiplicities of these constituents in $\pi$ imply that $a_W\geq 0$ for every $W$.
           Since $\pi$ is semisimple, $\bigoplus_W W^{\oplus a_W}$ is a real form of $\pi$.
          \end{proof}

        \begin{proof}[Proof of the Proposition \ref{AZ}]
        The map induced from the Aubert-Zelevinsky involution is an automorphism of order two on the set of the equivalence classes of irreducible representations of $G$.
        Hence, it suffices to show that it sends conjugate orthogonal representations to conjugate orthogonal ones.
        
        Let $\pi$ be a unitary conjugate orthogonal representation of $G$. Then, $\pi' = \ind^{G'}_G(\pi)$ is a unitary orthogonal representation of $G'$. 
        Hence, by Lemmas \ref{AZind} and \ref{AZorth}, $\pi'$ is defined over $\Rbb$. 
        Applying Lemmas \ref{AZind} and \ref{AZdefR} to $\pi'$, we conclude that $\mathrm{ind}_G^{G'}(|D_G(\pi)|)$ is a multiplicity-free unitary representation and belongs to the Grothendieck group of the category of representations of finite length and defined over $\Rbb$. 
        Then, Lemmas \ref{ssdefR} and \ref{AZorth} show that $\mathrm{ind}_G^{G'}(|D_G(\pi)|)$ is defined over $\Rbb$ and orthogonal.
        Hence, $|D_G(\pi)|$ is conjugate orthogonal by Lemma 4.10.
        \end{proof}

   \subsection{An auxiliary lemma}
      
         When $G$ is a group, let $G^{\op}$ denote its opposite group.
         We use the following trick.
         
         \begin{lemma}
         \label{trick}
         
         Let $\pi$ be a conjugate self-dual representation with respect to $(G, G')$. Then $\pi^{op}$ defined by $\pi^{op}(g) = \pi(g^{-1})$ is conjugate self-dual with respect to $(G^{\op}, G'^{\op})$, and has the same parity as $\pi$.
         
         \end{lemma}
          
         \begin{proof}
          It easily follows from the definition of parity.
          \end{proof}

\section{Proof of the main theorem}

        In this section, we finally prove the main theorem.  
        We first prove the following product formula for the parity of automorphic conjugate self-dual representations.
        We use the notation of Remark \ref{globalaux} which is a generalization of Construction \ref{global}.

    \subsection{A product formula for parity} 
         
        \begin{proposition}
        \label{prodform}
        Let $L/K$ be a quadratic extension of number fields, and let $(\mathbf{A}, \mathbf{B})$ be as in Remark~\ref{globalaux}.
        Assume that $\mathbf{A}$ is split at every archimedean place of $L$, and let $(G,G')$ be the associated conjugating pair.
        Let $\Pi = \bigotimes_{w} ' \Pi_{w}$ be an automorphic conjugate self-dual representation of $\mathbf{A}^{\times}(\mathbb{A}_L)$ which appears in the discrete spectrum with a unitary central character $\omega$. 
        Let also $S$ be the set of places at which $\mathbf{B}$ does not split and put $\Pi_{v} = \bigotimes_{w|v} \Pi_w$ for each place $v$ of $K$. 
        Then we have
        $$ \prod_{v \in S} c(G_v, G'_v, \Pi_{v}) = 1 .$$
         
        \end{proposition}
        
        \begin{proof}
        If we can show that $\Pi$ is conjugate orthogonal for the pair $(G, G')$, then this proposition follows from Lemmas \ref{unramified}, \ref{split}, and \ref{multi}. 
        The desired orthogonality results from the same argument as \cite[Section 3]{Prasad} using the multiplicity one theorem (Theorem \ref{BJL}). We give the detail.
        
        First, we take a $K$-valued point $\tau$ in $G' \setminus G$.
        Since $\Pi$ appears in the discrete spectrum with multiplicity one, we can consider $\Pi$ as a space of functions on $\mathbf{A}^{\times}(L)\backslash \mathbf{A}^{\times}(\Abb_L)$. This space has an invariant bilinear form, namely,
       
        $$ \langle f_1, f_2 \rangle 
        = \int_{Z_{\mathbf{A}^{\times}} (\Abb_L)\mathbf{A}^{\times}(L) \backslash \mathbf{A}^{\times}(\Abb_L)}f_1(\Int(\tau)(g))f_2(g)dg.$$

        We write $Tf=f\circ\Int(\tau)$.
        The space $T\Pi$ realizes $\Pi^\tau$ in the discrete spectrum with central character $\omega^{-1}$.
        The space $\overline{\Pi}$ realizes $\Pi^\vee$ because $\Pi$ is unitary.
        Since $\Pi^\tau\cong\Pi^\vee$, multiplicity one gives $T\Pi=\overline{\Pi}$.
        For every nonzero $f_1\in\Pi$, we take $f_2=\overline{Tf_1}\in\Pi$ and obtain $\langle f_1,f_2\rangle=\|Tf_1\|_2^2>0$.
        For every nonzero $f_2\in\Pi$, we take $f_1\in\Pi$ such that $Tf_1=\overline{f_2}$ and obtain $\langle f_1,f_2\rangle=\|f_2\|_2^2>0$.
        Thus the pairing is nondegenerate.
        It is invariant under the action on $\Pi^\tau\times\Pi$ by right invariance of the quotient measure.
        Conjugation by $\tau$ preserves this measure, since its square acts on the quotient by right translation by $\tau^{-2}$.
        We compute the parity by changing variables by $\Int(\tau)$ and using the left invariance of $f_2$ under $\tau^2\in\mathbf{A}^{\times}(L)$.
        
        \begin{align*}
        \langle \Pi(\tau^2)f_2, f_1 \rangle 
        &= \int_{Z_{\mathbf{A}^{\times}} (\Abb_L)\mathbf{A}^{\times}(L) \backslash \mathbf{A}^{\times}(\Abb_L)}f_2(\tau g\tau)f_1(g)dg \\
        &= \int_{Z_{\mathbf{A}^{\times}} (\Abb_L)\mathbf{A}^{\times}(L) \backslash \mathbf{A}^{\times}(\Abb_L)}f_2(\tau^2g)f_1(\Int(\tau)g)dg  \\
        &= \int_{Z_{\mathbf{A}^{\times}} (\Abb_L)\mathbf{A}^{\times}(L) \backslash \mathbf{A}^{\times}(\Abb_L)}f_2(g)f_1(\Int(\tau)g)dg  \\
        &= \langle f_1, f_2 \rangle 
        \end{align*}
        \end{proof}

	\subsection{The proof of the main theorem}
		
		We use the notation of Constructions \ref{locext}, \ref{global}.

		\begin{theorem}
		Let $E/F$ be a $p$-adic quadratic extension and $A=M_m(D)$ be a central simple algebra over $E$ of rank $n $ whose Hasse invariant is $s/d$. We take the pair $(s, d)$ in the way that $d > 1$, $0 < s < d$ and $\gcd(d, s)=1$, or $d=1$ and $s=0$.
		Let $(G, G')$ be the conjugating pair associated with it in Construction~\ref{locext}. Let $\pi$ be a conjugate self-dual discrete series representation of $G$ = $A^{\times}$ and let $\sigma$ denote its Langlands parameter. 
		Then we have
		$$ c(G, G',\pi) = (-1)^{(n-1)ms} c(W_E, W_F, \sigma)^{ms} .$$
		\end{theorem}
		
		\begin{proof}
	
		If $d=1$, the formula follows from Lemma~\ref{split}(1), since $s=0$.
		We now suppose that $d>1$ and apply Lemma~\ref{globalizefld} to globalize $E/F$ to a totally imaginary quadratic extension $L/K$ of a totally real field $K$.
		Then we take $ms$ finite places $v_1, \ldots, v_{ms}$ which are inert in $L/K$, and construct $\mathbf{D}$ and $\mathbf{E}$ as Construction \ref{global}.
		
		We then take conjugate self-dual simple supercuspidal representations $\pi_{i}$ of $\mathbf{D}_{v_i}^{\times}$ for $i = 1, \ldots, ms$, whose Langlands parameters have the same parity as $\sigma$ (see Theorem \ref{mie} (ii)). 
		Applying Theorem \ref{globalizerep} to the images of $\pi, \pi_{1}, \ldots, \pi_{ms}$ by the local Jacquet-Langlands correspondence, we get a conjugate self-dual cuspidal automorphic representation $\Pi'$ of $\GL_n(\Abb_L)$. 
		At every place where $\mathbf{D}$ does not split, the local component of $\Pi'$ is in the discrete series by construction.
		We apply Theorem~\ref{BJL} to obtain the transfer $\Pi$ of $\Pi'$.
		Note that the strong multiplicity one theorem for $\mathbf{A}^{\times}(\mathbb{A}_L)$ is valid (Theorem \ref{BJL}), and the Badulescu-Jacquet-Langlands correspondence is identity at split places, so $\Pi$ is conjugate self-dual.
    By applying the product formula above and Proposition \ref{AZ}, we have the equality
    $$ c(G, G',\pi) = \prod_{i=1}^{ms} c(G_{v_i}, G'_{v_i},\pi_{i}) .$$
    Furthermore, we apply Theorem \ref{mie}, Proposition \ref{AZ} and Lemma \ref{trick} to obtain 
    $$ c(G_{v_i}, G'_{v_i},\pi_{i}) = (-1)^{n-1} c(W_E, W_F, \sigma) .$$
    Substituting this to the first equality, we have 
    $$ c(G, G',\pi) = ((-1)^{n-1} c(W_E, W_F, \sigma))^{ms} .$$ 
    This is the very statement of our main theorem.
		\end{proof}
		
		\begin{remark}
		The same argument works if we replace $A'$ with the central division algebra $A''$ with Hasse invariant $(s+d)/2d$. We write $(G, G'')$ for the associated pair in this case.
    Then, we obtain the equality
		 \begin{align*}
		 c(G, G'',\pi)
		 &=(-1)^{(n-1)m(s+d)} c(W_E, W_F, \sigma)^{m(s+d)} \\
		 &=c(W_E, W_F, \sigma)^n c(G, G',\pi).
         \end{align*}
		
		By the inflation-restriction sequence for $\Gbb_m$, there are only these two possibilities for $A'$ which satisfy the hypothesis of Lemma \ref{keylemma1}. The occurrence of this parity change comes from the central character of a representation. See \cite[Remark 2.2]{Mie}.
		\end{remark}

%	References
\vskip\baselineskip
\bibliographystyle{plain}
\bibliography{bibliography}

\end{document}